\documentclass[10pt, english]{smfart}
\usepackage{amssymb}
\usepackage{amsmath}
\usepackage{amsthm}
\usepackage{mathrsfs}
\usepackage{graphicx}
\usepackage{color}
\usepackage{manfnt}
\usepackage[utf8]{inputenc}   
\usepackage{geometry}         
\usepackage[english]{babel}  
\usepackage{verbatim}

\usepackage[all]{xy}
\usepackage{hyperref}

\newtheorem{thm}{Theorem}[section]
\newtheorem{cor}[thm]{Corollary}

\theoremstyle{definition}
\newtheorem{rem}[thm]{Remark}
\newtheorem{remarks}[thm]{Remarks}

\newtheorem{defn}[thm]{Definition}

\newtheorem*{thms}{Theorem}

\theoremstyle{remark}

\numberwithin{equation}{section}

\newcommand{\abs}[1]{\left\vert#1\right\vert}

\newcommand{\ord}{\text{ord}}
\newcommand{\ac}{\overline{\text{ac}}}

\newcommand{\Spec}{\text{Spec}}
\newcommand{\A}{\mathbb A}
\newcommand{\PP}{\mathbb P}
\newcommand{\C}{\mathbb C}

\newcommand{\Z}{\mathbb Z}

\renewcommand{\L}{\mathbb L}

\newcommand{\ra}{\rightarrow}

\newcommand{\Var}{\text{Var}}

\newcommand{\mes}{\text{mes}}

\newcommand{\an}{\mathrm{an}}
\renewcommand{\k}{k}
\definecolor{vert}{rgb}{0,0.6,0.2}

\newcommand{\F}{\mathcal F}

\newcommand{\Ubeth}{\big(\widehat{U_R}\big)^\beth}
\newcommand{\mot}{\mathrm{mot}}

\newcommand{\Vol}{\text{Vol}}

\newcommand{\X}{\mathcal{X}}

\title{Motivic and analytic nearby fibers at infinity and bifurcation sets}

\author{Lorenzo Fantini}
\address{Aix-Marseille Université, Institut de Mathématiques de Marseille, 13453 Marseille, France}
\email{\href{mailto:lorenzo.fantini@univ-amu.fr}{lorenzo.fantini@univ-amu.fr}}
\urladdr{\url{https://www.i2m.univ-amu.fr/perso/fantini.l/}}

\author{Michel Raibaut} 
\address{Laboratoire de Math\'ematiques\\ Univ. Grenoble Alpes, Université Savoie Mont Blanc, B\^atiment Chablais, Campus Scientifique, 
Le Bourget du Lac, 73376 Cedex, France} 
\email{\href{mailto:michel.raibaut@univ-smb.fr}{michel.raibaut@univ-smb.fr}}
\urladdr{\url{www.lama.univ-savoie.fr/~raibaut}}

\begin{document}

\begin{abstract} 
	In this paper we use motivic integration and non-archimedean analytic geometry to study the singularities at infinity of the fibers of a polynomial map $f\colon \A^d_\C \to \A^1_\C$.
	We show that the motive $S_{f,a}^{\infty}$ of the motivic nearby cycles at infinity of $f$ for a value $a$ is a motivic generalization of the classical invariant $\lambda_f(a)$, an integer that measures a lack of equisingularity at infinity in the fiber $f^{-1}(a)$. 
	We then introduce a non-archimedean analytic nearby fiber at infinity $\mathcal F_{f,a}^{\infty}$ whose motivic volume recovers the motive $S_{f,a}^{\infty}$. 
	With each of $S_{f,a}^{\infty}$ and $\mathcal F_{f,a}^{\infty}$ can be naturally associated a bifurcation set; we show that the first one always contains the second one, and that both contain the classical topological bifurcation set of $f$ if $f$ has isolated singularities at infinity. 
\end{abstract}

\maketitle

\today
\tableofcontents

\section{Introduction}

Let $U$ be a smooth and irreducible complex algebraic variety and let $f\colon U \to \A^1_\C$ be a dominant map.
It is well known that there exists a finite subset $B$ of $\A^1_\C(\C)$ such that the map
\[
f|_{U \setminus f^{-1}(B)} \colon U \setminus f^{-1}(B) \longrightarrow \A^1_\C \setminus B
\]
is a locally trivial $C^\infty$-fibration, see for example \cite{Pha83a}.
The smallest such subset $B$ is called the \emph{topological bifurcation set} of $f$ and is denoted by $B_f^\mathrm{top}$.
The topological bifurcation set of $f$ contains the discriminant $\mathrm{disc}(f)$ of $f$, that is the set $f\big(\mathrm{crit}(f)\big) \subset \A^1_\C(\C)$ image of its set of critical points, but the inclusion might be strict since $f$ needs not to be proper.
For example, the map $f \colon \A^2_\C \to \A^1_\C$ defined by the polynomial $f(x,y)=x(xy-1)$ is smooth, but has bifurcation set $\{0\}$; its fiber over $0$ is the only one to be disconnected.
This is related to the fact that, loosely speaking, $f(x,y)$ may have some ``critical points at infinity''; one way of seeing this is to observe that the compactification of the fiber $f^{-1}(a)$ in $\PP^2_\C$ will have a multibranch singularity for $a = 0$, and a cuspidal singularity otherwise.

Assume that $f \colon \A^d_\C \to \A^1_\C$ is a polynomial map.
In some special cases, knowing the (compactly supported) Euler characteristic of a fiber $f^{-1}(a)$ is sufficient to determine whether $a$ belongs to $B_f^\mathrm{top}$.
Indeed, if we denote by $f^{-1}(a_\mathrm{gen})$ a general fiber of $f$, we have 
\begin{equation}\label{eq_Bftop_Euler_char}
B_{f}^\mathrm{top} = \big\{ a \in \A^1_\C(\C)  \;\big|\; \chi_c\big(f^{-1}(a)\big) \neq \chi_c\big(f^{-1}(a_\mathrm{gen})\big)\big\}
\end{equation}
in the case of curves, that is whenever $d=2$, see \cite{VuiTra84a}, and, more generally, whenever $f$ has \emph{isolated singularities at infinity}  (that is, the closure of a general fiber $f^{-1}(a_\mathrm{gen})$ of $f$ in $\PP^d_\C$ has isolated singularities), see \cite{Par95a}.

If the polynomial $f\in\C[x_1,\ldots,x_d]$ has isolated singularities in $\A^d_\C$ (which does not necessarily mean that it has also isolated singularities at infinity), then Artal Bartolo--Luengo--Melle-Hern{{\'a}}ndez proved in \cite{ArtLueMel00b} that the Euler characteristic of each fiber $f^{-1}(a)$ can be computed in terms of the Euler characteristic of $f^{-1}(a_\mathrm{gen})$ and of some local numerical invariants, as we will now explain.
To account for the singularities of the fibers, let $\mu_a(f)$ be the sum of the Milnor numbers of the singular points of $f^{-1}(a)$, so that $a$ belongs to the discriminant of $f$ if and only if $\mu_a(f)>0$, and set $\mu(f) =\sum_{a \in \mathbb A^1_{\mathbb C}} \mu_a(f)$.
On the other hand, to keep track of the behavior of $f$ at infinity, consider $\lambda_a(f) = \mu\big(\mathbb P^n, \overline{f^{-1}(a)},D\big) - \mu\big(\mathbb P^d, \overline{f^{-1}(a_\mathrm{gen}}),D\big)$, where $\overline{f^{-1}(a)}$ is the closure $f^{-1}(a)$ in $\mathbb P^d$, $D$ is the divisor cut out by the homogeneous part of $f$ of degree $\mathrm{deg}(f)$ in the hyperplane at infinity of $\mathbb P^d$, and $\mu\big(\mathbb P^n, \overline{f^{-1}(b)},D\big)$ is the generalized Milnor number of Parusi\'nski from \cite{Par88}.
The invariants $\lambda_a(f)$ vanish for all but finitely many values of $a$, therefore we can set $\lambda(f) =\sum_{a \in \mathbb A^1_{\mathbb C}} \lambda_a(f)$.
By \cite[Theorem 1.7]{ArtLueMel00b}, we then have the following equalities:
	\begin{align} 
	& \chi_{c}\big(f^{-1}(a)\big) = \chi_c\big(f^{-1}(a_\mathrm{gen})\big) + (-1)^{d}\big(\mu_a(f) + \lambda_a(f)\big),
	\label{eq:lambda} \\
	& \chi_c\big(f^{-1}(a_\mathrm{gen})\big) = 1 + (-1)^{d-1}\big(\mu(f) + \lambda(f)\big).
	\label{eq:caracteristiquefibregenerique}
	\end{align}
Similar results were obtained in \cite{Tib07a} and \cite{Par97a}, and in the case of curves in \cite{Suz74}, \cite{VuiTra84a}, and \cite{Cas96}.
In particular, whenever the condition \ref{eq_Bftop_Euler_char} also holds one can deduce that
	\[
	B_{f}^\mathrm{top} = \big\{a\in \mathbb A^1_{\C}(\C) \;\big|\; \lambda_a(f) \neq 0\big\} \cup \text{disc}(f)\;.
	\]

The invariant $\lambda_a(f)$ gives some rudimentary measure of the lack of equisingularity of $f$ at infinity at the value $a$.
Since it is defined using Euler characteristics, it is natural to expect to be able to generalize $\lambda_a(f)$ to a motive in a Grothendieck ring of $\C$-varieties by using Denef and Loeser's motivic integration.
In that direction, and with the goal of studying singularities at infinity, the second author defined in \cite{Rai11} a motive $S_{f,a}^{\infty}$, the so called \emph{motivic nearby cycles at infinity of $f$ for the value $a$}, living in a localization $\mathcal M_{\C}^{\hat{\mu}}$ of the Grothendieck ring of $\C$-varieties with a good $\hat{\mu}$-action.
In order to do so, he used motivic integration techniques, following work of Denef--Loeser \cite{DenLoe98b}, Bittner \cite{Bit05a}, and Guibert--Loeser--Merle \cite{GuiLoeMer06a}.
This allowed him to define a \emph{motivic bifurcation set} $B_f^\mathrm{mot}$ of $f$ as the union of the discriminant of $f$ with the set of values whose motivic nearby cycles at infinity do not vanish.
This is a finite set, and it coincides with $B_f^\mathrm{top}$ if $f$ is a convenient and non degenerate polynomial with respect to its Newton polygon at infinity (in the sense of \cite{Kou76a}), as is shown in \cite[Theorem 4.8]{Rai11}.
However, no explicit link between $S_{f,a}^\infty$ and $\lambda_a(f)$ was established in \cite{Rai11}.

The first contribution of this paper is to prove that $S_{f,a}^{\infty}$ is indeed a motivic generalization of the numerical invariant $\lambda_a(f)$, whenever the latter is defined.
In particular, we obtain the following result:
\begin{thms} 
	Let $f$ be a polynomial in $\C[x_1,\dots,x_d]$ with isolated singularities in $\A^d_\C$. 
	For any value $a$ in $\A^1_\C(\C)$, we have 
	\begin{equation*} \label{thm-intro:lambda}
		\chi_{c}\left(S_{f,a}^{\infty}\right) = (-1)^{d-1}\lambda_a(f)\,.
	\end{equation*}
\end{thms}

We obtain this result as a consequence of a more general equality between {Euler characteristic of} motives, proven in Theorem~\ref{thm:lambda}, which in particular does not require $f$ to have isolated singularities.
As an application, we deduce the following fact:
\begin{thms} 
Let $f$ be a polynomial in $\C[x_1,\dots,x_d]$. 
Then we have
\begin{equation*}
\big\{ a \in \A^1_\C(\C)  \;\big|\; 
\chi_c\big(f^{-1}(a)\big) \neq \chi_c\big(f^{-1}(a_\mathrm{gen})\big)
\big\}
\subset
B_f^\mathrm{mot}\;.
\end{equation*}
\end{thms}
In particular, $B_f^\mathrm{mot}$ contains $B_f^\mathrm{top}$ whenever the condition \ref{eq_Bftop_Euler_char} is satisfied, for example if $f$ has isolated singularities at infinity.

Let us briefly explain the definition of $S_{f,a}^{\infty}$ and give a short outline of the proof of the theorems.
We can work in a slightly more general setting and assume that $k$ is a field of characteristic zero containing all roots of unity, $U$ is a smooth $k$-variety, and $f\colon U\to \A^1_k$ is a dominant morphism.
We consider a compactification $(X,\hat{f})$ of $(U,f)$, that is a $k$-variety $X$ containing $U$ as a dense open subset, together with a proper extension $\hat{f}$ of $f$ to $X$.
As is usual in motivic integration, the motive $S_{f,a}^{\infty}$ is defined as a limit of a generating series $\sum \mes(X_i)T^i$, where $\mes(X_i)$ is the motivic measure of a subset $X_i$ of the scheme of arcs $\mathcal L(X)$ of $X$.
Namely, $X_i$ consists of arcs that have origin at infinity, that is in {$X\setminus U$}, order of contact $i$ along the fiber $\hat{f}^{-1}(a)$, and are only mildly tangent to {$X\setminus U$}.
As such, $S_{f,a}^{\infty}$ can be expressed as the difference of the images in $\mathcal M_k^{\hat{\mu}}$ of $S_{\hat{f}-a,U}$, the motivic nearby cycles of $\hat{f}-a$ supported by $U$, a motive defined by Guibert--Loeser--Merle \cite{GuiLoeMer06a} using a motivic integration procedure as above but imposing no condition on the origin of arcs, by the motivic nearby fiber $S_{f-a}$ of Denef--Loeser \cite{DenLoe98b}, which is constructed from arcs with origin in $U$.
We show in Theorem~\ref{thm:lambda} that the resulting equality generalizes the formula \ref{eq:lambda}, by taking the Euler characteristics and using realization results for $S_{\hat{f}-a,U}$ and $S_{f-a}$ from \cite{GuiLoeMer06a} and \cite{DenLoe98b}.
The two theorems we stated above follow then from simple computations.

The second contribution of this paper is the construction of a non-archimedean analytic version of $S_{f,a}^{\infty}$.
Namely, the \emph{analytic nearby fiber $\mathcal F_{f,a}^\infty$ of $f$ for the value $a$} is the analytic space over the field $K=k((t))$ (endowed with a given $t$-absolute value) defined as
\[
\F_{f,a}^{\infty} = (U_K)^{\an} \setminus \Ubeth\;,
\]
where $\Ubeth$ is the space of analytic nearby cycles of \cite{NicSeb07a}, that is the analytic space associated to the formal completion of $U$ along the fiber $f^{-1}(a)$, and $(U_K)^\an$ is the analytification of the base change of $U$ to $K$ via the map $f-a$.
This definition is analogous to the one of the motive $S_{f,a}^{\infty}$, since if $(X,\hat{f})$ is a compactification of $(U,f)$ then $\F_{f,a}^{\infty}$ consists of those points in the analytification $(X_K)^\an$ of $X_K$ that have origin on $X_K\setminus U_K$ but are not completely contained in $X_K\setminus U_K$.
One advantage of $\F_{f,a}^\infty$ over $S_{f,a}^{\infty}$ is that, while the latter is defined in terms of an extension of $f$ to a compactification of $U$, and only later proven to be independent of such a choice, the former is defined exclusively in terms of $f$.

We then construct a motivic specialization of $\mathcal F_{f,a}^\infty$, its \emph{Serre invariant} $\overline{S}\big(\mathcal F_{f,a}^{\infty}\big)$, and use it to define a third bifurcation set, the \emph{Serre bifurcation set} $B_f^{\mathrm{ser}}$ of $f$, as the union of the discriminant of $f$ with the set of values $a$ such that $\overline{S}\big(\mathcal F_{f,a}^{\infty}\big)$ does not vanish.

Using the decomposition of $\F_{f,a}^\infty$ as a difference, and realization results for the volume and Serre invariant of analytic spaces from \cite{NicSeb07a,Bultot-Phd,Hartmann}, in the subsections \ref{subsection_analyticinfinity} and \ref{subsection_serrebifurcation} we prove the following results:

\begin{thms} 
Let $U$ be a smooth connected $k$-variety and let $f\colon U\to \A^1_k$ be a dominant morphism.
For any value $a$ in $\A^1_k(k)$, we have:
\begin{enumerate}
	\item 
\(
\Vol\big(\mathcal F_{f,a}^{\infty}\big) = \mathbb L^{-(\dim U-1)} S_{f,a}^{\infty}
\)
in $\mathcal M_{k}^{\hat{\mu}}$;
\item
\(
\overline{S}\big(\mathcal F_{f,a}^{\infty}\big) = S_{f,a}^{\infty} \mod (\mathbb L-1)
\)
in $\mathcal M_{k}^{\hat{\mu}}/(\mathbb L-1)$;
\item We have the inclusions
$$\big\{ a \in \A^1_\C(\C)  \;\big|\; 
\chi_c\big(f^{-1}(a)\big) \neq \chi_c\big(f^{-1}(a_\mathrm{gen})\big)
\big\}
\subset B_f^{\mathrm{ser}} \subset B_f^\mathrm{mot}\;.$$
\end{enumerate}\end{thms}

We expect that a deep study of the geometry (and étale cohomology) of the analytic nearby fibers at infinity will lead to a better understanding of various phenomena of equisingularity at infinity, and we plan to study this topic further in an upcoming project.

\subsection*{Acknowledgments}
We are very thankful to David Bourqui, Raf Cluckers, Johannes Nicaise, and Julien Sebag, who organized the conference ``Nash : Schémas des arcs et singularités'' in Rennes in 2016, gave the first of us the opportunity to give a talk there, and proposed us to contribute to the conference proceedings.   
We are also grateful to Emmanuel Bultot, Pierrette Cassou-Noguès, Alexandru Dimca, Johannes Nicaise, and Claude Sabbah for inspiring discussions, and the anonymous referee for his comments and corrections.
This work is partially supported by ANR-15-CE40-0008 (Défigéo).

\subsection*{Notations}

In the paper we will freely use the following notations.

\begin{itemize} 
	\item $k$ is a field of characteristic zero that contains all the roots of unity.
	\item A $k$-variety is a separated $k$-scheme of finite type.
	\item $\hat{\mu}$ is the projective limit $\underset{\leftarrow}{\lim}\mu_n$, where for any positive integer $n$ we denote by $\mu_n$ the group scheme of $n$-th roots of unity.
	\item $R$ is the discrete valuation ring $k[[t]]$, endowed with the t-adic valuation. 
	\item $K=k((t))$ is the fraction field of $R$, endowed with a fixed $t$-adic absolute value. 
	\item For any integer {$n$}, we denote by $K(n)=K((t^{1/n}))$ the unique extension of degree $n$ of $K$ obtained by joining a $n$-th root of $t$ to $K$, and by $R(n)$ the normalization of $R$ in $K(n)$.
	\item Given a $k$-variety $X$ and a morphism $f\colon X\to \A^1_k=\Spec\,k[t]$, we set $X_R=X\times_{\A^1_k}\Spec\,R$ and $X_K=X\times_{\A^1_k}\Spec\,K$.
\end{itemize}


\section{Motivic integration and nearby cycles} \label{section:rappels}
\label{section_motivicint}

In this section we review some basic constructions in motivic integration that will be used throughout the paper.  
We refer to \cite{DenLoe99a}, \cite{DenLoe01b}, \cite{Loe09a}, \cite{Loo02a}, \cite{GuiLoeMer06a}, and \cite{GuiLoeMer05a} for a more thorough discussion of these notions.

\subsection{Grothendieck rings} \label{inverse-direct}
We say that an action of $\mu_n$ on a $k$-variety $X$ is \emph{good} if every ${\mu_n}$-orbit is contained in some open affine subscheme of $X$, and that an action of $\hat{\mu}$ on $X$ is \emph{good} if it factors through a good $\mu_n$-action for some $n$.
If $S$ is a $k$-variety, we denote by $\Var^{\hat{\mu}}_S$ the category of $S$-varieties with a good $\hat{\mu}$-action, that is the category whose objects are the $k$-varieties $X$ endowed with a good action of $\hat{\mu}$ 
and with a morphism $X\to S$ that is equivariant with respect to the trivial $\hat{\mu}$-action on $S$, and whose morphisms are $\hat{\mu}$-equivariant morphisms over $S$.

We denote by $K_0(\Var_{S}^{\hat{\mu}})$ the Grothendieck ring of $\Var_{S}^{\hat{\mu}}$.
It is defined as the abelian group generated by the isomorphism classes $[X,\sigma]$ of the elements $(X,\sigma)$ of $\Var_{S}^{\hat{\mu}}$, with the relations $[X,\sigma] = [Y,\sigma|_Y] + [X \setminus Y,\sigma|_{X \setminus Y}]$ if $Y$ is
a closed subvariety of $X$ stable under the action of $\hat{\mu}$, and moreover $[X \times \mathbb A^n_{k}, \sigma] = [X \times \mathbb A^{n}_{k} , \sigma']$ if $\sigma$ and $\sigma'$ are two liftings of the same $\hat{\mu}$-action on $X$ to an affine action on $X \times \mathbb A^k_n$. 
There is a natural ring structure on  $K_0(\Var_{S}^{\hat{\mu}})$, the product being induced by the fiber product over $S$.
In the rest of the paper we will simply write $[X]$ for $[X,\sigma_X]$, as no risk of confusion will arise.

We denote by $\mathcal M_S^{\hat{\mu}}$ the localization of $K_0(\Var_{S}^{\hat{\mu}})$ at the element $\mathbb L_S = [\mathbb A^1_S]$, that is the class of the affine line over $S$ endowed with the trivial action. 
If $S=\Spec \, k$ we simply write $\mathcal M_k^{\hat{\mu}}$ and $\mathbb L$ for $\mathcal M_S^{\hat{\mu}}$ and $\mathbb L_S$ respectively, and if we only consider varieties with trivial $\hat{\mu}$-action, we obtain analogous Grothendieck rings $\mathcal M_k$ and $\mathcal M_S$.

If $f:S'\ra S$ is a morphism of $k$-varieties, then the composition with $f$ on the left induces a \emph{direct image} group morphism
\[
f_{!}\colon\mathcal M_{S'}^{\hat{\mu}} \ra \mathcal M_{S}^{\hat{\mu}}\,,
\]
while taking the fiber product with $S'$ over $S$ induces an \emph{inverse image} ring morphism
\[
f^{*}\colon\mathcal M_{S}^{\hat{\mu}} \ra \mathcal M_{S'}^{\hat{\mu}}\,.
\]

We denote by $\mathcal M_{S}^{\hat{\mu}}[[T]]_\mathrm{rat}$ the $\mathcal M_{S}^{\hat{\mu}}$-submodule of $\mathcal M_{S}^{\hat{\mu}}[[T]]$ generated by 1 and by the finite products of terms $p_{e,i}(T)=\mathbb L^{e}T^{i}/(1-\mathbb L^{e}T^{i})$, with $e$ in
$\Z$ and $i$ in $\Z_{>0}$. 
The formal series in $\mathcal M_{S}^{\hat{\mu}}[[T]]_\mathrm{rat}$ are said to be \emph{rational}.
There exists a unique $\mathcal M_{S}^{\hat{\mu}}$-linear morphism
$\lim_{T \ra \infty} \colon \mathcal M_{S}^{\hat{\mu}}[[T]]_\mathrm{rat} \to \mathcal M_{S}^{\hat{\mu}}$ such that for any finite subset $(e_i,j_i)_{I\in i}$ of $\mathbb Z \times \Z_{>0}$ we have $\lim_{T \ra \infty} \big( \prod_I p_{e_i,j_i}(T)\big)=(-1)^{\abs{I}}$.

\subsection{Arcs on varieties} 

Let $X$ be a $\k$-variety of dimension $d$. 
We denote by $\mathcal L_{n}(X)$ the \emph{space of $n$-jets} of $X$, that is the $\k$-scheme of finite type whose functor of points is the following: for every $k$-algebra $L$, the $L$-points of $\mathcal L_n(X)$ are the morphisms $\Spec \big(L[s]/(s^{n+1})\big) \to X$. 
In particular, $\mathcal L_0(X)$ is canonically isomorphic to $X$ itself.
Right compositions with the canonical projections $L[s]/(s^{n+1})\to L[s]/(s^{n})$ yield morphisms $\mathcal L_{n}(X)\ra \mathcal L_{n-1}(X)$, making $\{\mathcal L_n(X)\}_n$ into a projective system. 
These morphisms are $\mathbb A_{\k}^{d}$-bundles when $X$ is smooth of pure dimension $d$.
The \emph{arc space} of $X$, denoted by  $\mathcal L(X)$, is the projective limit of the system $\{\mathcal L_n(X)\}_n$; we denote by $\pi_{n} \colon \mathcal L(X)\ra \mathcal L_{n}(X)$ the canonical projections. 
The arc space of $X$ is a $\k$-scheme that is generally not of finite type.
For every finite extension $k'$ of $k$ the $k'$-rational points of $\mathcal L(X)$ parametrize the morphisms $\Spec\,k'[[s]] \to X$. 
If $\varphi\in\mathcal L(X)$ is an arc on $X$, the point $\pi_0(\varphi)\in X$ is called the \emph{origin} of $\varphi$ on $X$.
Observe that  ${\mu_n}$ acts canonically on $\mathcal L_{n}(X)$ and on $\mathcal L(X)$ by setting $\lambda.\varphi(s)=\varphi(\lambda s)$.

Assume that we have a morphism $f\colon X \to \A^1_k = \Spec \, k[t]$.
With a $L$-arc $\varphi \colon {\Spec\:} L[[s]] \to X$ on $X$ we can then associate its \emph{order} $\ord f(\varphi):= \ord_s \big(\varphi^\#(f)\big) \in \Z_{\geq0}\cup\{+\infty\}$, and its \emph{angular component} $\ac \, f(\varphi)\in L$, defined as the leading coefficient of the power series $\varphi^\#(f)\in L[[s]]$; by convention $\ac \, f(\varphi) = 0$ if $\varphi^\#(f)=0$.

More generally, if $F$ is a closed subvariety of $X$ of coherent ideal sheaf $\mathcal I_{F}$, then the \emph{contact order} of an arc $\varphi\in\mathcal L(X)$ along $F$ is $\ord_F(\varphi):= \inf_g \ord \; \varphi^\#(g)$, where $g$ runs among the local sections of $\mathcal I_{F}$ at the origin of $\varphi$.
It is greater than zero if and only if the origin of $\varphi$ lies in $F$, while it takes the value infinity if $\varphi$ is an arc on $F$.
For example, if we have a morphism $f\colon X \to \A^1_k$ as above then the contact order of $\varphi$ along the fiber $f^{-1}(0)$ is precisely $\ord f(\varphi)$.


\subsection{Motivic nearby cycles} \label{section:Sf} 
Let $X$ be a purely dimensional $k$-variety, let $f \colon X\ra \mathbb A^{1}_{\k}$ be a morphism, and denote by $X_{0}(f)$ the zero locus of $f$ in $X$.
By work of Denef--Loeser \cite{DenLoe98b,DenLoe02a}, Bittner \cite{Bit05a}, and Guibert--Loeser--Merle \cite{GuiLoeMer06a}, there exists a group morphism
\[
\mathcal S_{f}\colon\mathcal M_{X} \longrightarrow \mathcal M_{X_{0}(f)}^{\hat{\mu}}
\]
called the \emph{motivic nearby cycles morphism} of $f$.
In particular, for every open immersion $i\colon U\to X$ we obtain a motive $S_{f,U}:=S_f(i)$ in $\mathcal M_{X_{0}(f)}^{\hat{\mu}}$, the \emph{motivic nearby cycles of $f$ supported on $U$}. 

We will recall the construction of the motive $S_{f,U}$ as done in \cite{GuiLoeMer06a}, restricting to the case where $U$ is a smooth dense open subvariety of $X$, since this case will play an important role in the rest of the paper.
We denote by $F=X\setminus U$ the complement of $U$ in $X$, and by passing to a resolution of the singularities of the pair $(X,F)$ we can assume without loss of generality that $X$ is itself smooth. 
For any two positive integers $n$ and $\delta$, we consider the subscheme
\[
X_{n}^{\delta}(f,U):=\big\{ 
\varphi \in \mathcal L(X) \;\big|\; \ord\, f(\varphi)=n,\; \ac \, f(\varphi)=1,\:
\ord_F (\varphi) \leq  n \delta 
\big\}
\]
of $\mathcal L(X)$.
Since $\ord\, f(\varphi)>0$, we have that $\pi_{0}\big(X_{n}^{\delta}(f,U)\big)\subset X_{0}(f)$, and so for any $m\geq n$ the image $\pi_m\big(X_{n}^{\delta}(f,U)\big)\subset \mathcal L_m(X)$ is endowed with an action of $\mu_n$ given by $\lambda.\varphi(t)=\varphi(\lambda t)$, giving rise to  a class in $\mathcal M_{X_0(f)}^{\hat{\mu}}$; we denote this motive by $\big[X_{n,m}^{\delta}(f,U)\big]$.
Since $X$ is smooth we have an equality 
\[
\big[X_{n,m}^{\delta}(f,U)\big] \mathbb L^{-md} 
=
\big[X_{n,n}^{\delta}(f,U)\big] \mathbb L^{-nd} 
\in 
\mathcal M_{X_0(f)}^{\hat{\mu}}\,.
\] 
This element of $\mathcal M_{X_0(f)}^{\hat{\mu}}$ is called the \emph{motivic measure} of $X_{n}^{\delta}(f,U)$ and denoted by $\mes\big(X_{n}^{\delta}(f,U)\big)$.


For any $\delta$, consider the generating series
\[
Z_{f,U}^{\delta}(T)
:=
\sum_{n \geq 1} \mes\big( X_{n}^{\delta}(f,U)\big)T^{n} 
\in
\mathcal M_{X_{0}(f)}^{\hat{\mu}}[[T]].
\]
By giving an explicit formula for this power series in terms of an embedded resolution of $f$ and of $F$, Guibert--Loeser--Merle show that it is rational,  (see \cite[Proposition 3.8]{GuiLoeMer06a}).
One can then consider the limit $- \lim_{T\ra \infty} Z_{f,U}^{\delta}(T)$ as an element of $\mathcal M_{X_{0}(f)}^{\hat{\mu}}$.
Moreover, they prove that the limit does not depend on the choice of $\delta$, provided that it is big enough.
This limit is the motive $\mathcal S_{f,U}$ that we wanted to define.

\begin{remarks}\phantomsection\label{remarks_SfU}
	\begin{enumerate}
		\item Assume that $X$ is smooth and take $U=X$. 
		Then in the construction above the condition on the tangency of arcs to $F$ disappears, and there is therefore no need to show that the generating series does not depend on $\delta$.
		The resulting motive $S_f:=S_{f,X}$ is the \emph{motivic nearby cycles} defined by Denef--Loeser, and the procedure we sketched is the original construction from \cite{DenLoe98b,DenLoe02a}. 
		It follows from the formula expressing $S_f$ in terms of a log-resolution of $\big(X,f^{-1}(0)\big)$ that if $0$ is not in the discriminant of $f$ then $S_f=[f^{-1}(0)]$.
	\item \label{rem:additivity} It follows from the construction of Guibert--Loeser--Merle that if $X$ and $F$ are smooth then
		\[
		S_{f}\big([i_F \colon F\ra X]\big) = i_{F!}S_{f\circ i_F,F} = i_{F!}S_{f_{\mid F}}.
		\]
		In particular, by additivity of the morphism $S_f$ we obtain an equality 
		\[
		S_{f,U} = S_f - i_{F!}S_{f_{\mid F}}.
		\]
	\item Assume that $k$ is the field of complex numbers. 
	Bittner \cite[\S 8]{Bit05a} and Guibert--Loeser--Merle \cite[Proposition 3.17]{GuiLoeMer06a}, extending a result of Denef--Loeser \cite[Theorem 4.2.1]{DenLoe98b}, show that the motivic nearby cycles morphism is indeed a motivic version of the nearby cycles functor, which justifies the terminology.
	More precisely, they show that the following diagram is commutative:
	\[
		\xymatrix{ \mathcal M_X \ar^{S_f}[rr] \ar_{\chi_X}[d] && \mathcal M^{\hat{\mu}}_{X_{0}(f)} \ar^-{\chi_{X_0(f)}}[d] \\
		K_0\big(D_{c}^{b}(X)\big) \ar^-{\psi_f}[rr] && **[c] K_0\big(D_{c}^{b}(X_{0}(f))^{\mathrm{mon}}\big) \;, }
	\]
	where $K_0\big(D_{c}^{b}(X)\big)$ and $K_0\big(D_{c}^{b}(X_{0}(f))^{\mathrm{mon}}\big)$ are the Grothendieck rings of the derived categories of bounded constructible complex sheaves on $X$ and $X_{0}(f)$ respectively (the latter being endowed with a quasi-unipotent action), $\psi_f$ is the nearby cycle functor at the level of sheaves (see \cite{SGA7}), and $\chi_X$ and $\chi_{X_0(f)}$ are the realization morphisms which assign to a variety $p \colon Y\ra X$ or $\big(p \colon Y\ra X_0(f),\sigma \big)$ the corresponding class $[Rp_!\underline{\mathbb{Q}}_Y]$ and $[Rp_!\underline{\mathbb Q}_Y,\sigma]$ respectively.
	Moreover, the direct image functors at the level of motives and at the level of sheaves are compatible. 
	For instance, if $i\colon U\to X$ is an open immersion and $f\colon X\to \mathbb A^1_\mathbb C$ is a morphism, then we have
	\[
	\chi_{X_0(f)}\big(S_{f,U}\big) = \Big[\psi_f\big(Ri_{!}\underline{\mathbb Q}_U\big)\Big]\;\;
	\]
	and 
	\[
	\chi_{\mathbb C}\big({f}_{!}S_{f,U}\big) = \Big[R_{f!}\psi_f\big(Ri_{!}\underline{\mathbb Q}_U\big)\Big]\,.
	\]
	\end{enumerate}
\end{remarks}


\section{Motivic nearby cycles at infinity and the motivic bifurcation set}
\label{section_motivicinfinity}

In this section we study the motivic nearby cycles at infinity $S_{f,a}^{\infty}$ for a fiber $f^{-1}(a)$ of a morphism $f\colon U\to \A^1_k$, a motive introduced by the second author in \cite{Rai11}.
We prove that this object is a motivic generalization of the invariants $\lambda_a(f)$ mentioned in the introduction.
We then deduce that under a natural assumption, that is when $f$ has isolated singularities at infinity, the motivic bifurcation set of $f$, which is defined in \emph{loc. cit.} as the set of those values $a$ such that $S_{f,a}^{\infty}$ doesn't vanish, contains the topological bifurcation set of $f$.

\subsection{Motivic nearby cycles at infinity} \label{subsection:motivic-invariant-infinity}
We begin by recalling some constructions of the second author from \cite[\S 4]{Rai11}.
Let $U$ be a smooth connected algebraic $k$-variety and let $f\colon U\ra \mathbb A^1_k$ be a dominant morphism.

By Nagata compactification theorem there exists a \emph{compactification} $(X,i,\hat{f})$ of $f$, by which we mean the data of a $k$-variety $X$, an open dominant immersion $i\colon U\ra X$, and a proper map $\hat{f}\colon X \ra \mathbb A^{1}_{k}$ such that the following diagram is commutative:
\[
\xymatrix{ U \ar^{i}[r] \ar_{f}[dr] & X \ar^{\hat{f}}[d] \\
           & \mathbb A^{1}_{k}
}
\]

In the following, we will identify $U$ with its image $i(U)$ in $X$, drop the immersion $i$ from the notation $(X,i,\hat{f})$, and denote by $F$ be the closed subvariety $X\setminus U$ of $X$.

\begin{defn}[Motivic nearby cycles at infinity] For any $a$ in $\mathbb A_{\k}^{1}(k)$, we call \emph{motivic nearby cycles at infinity of $f$ for the value $a$} the motive
	\[
	S_{f,a}^{\infty} = \hat{f_!}\big(S_{\hat{f}-a,U} - i_!S_{f-a}\big) \in \mathcal M_{k}^{\hat{\mu}}.
	\]
\end{defn}

It is shown in \cite[Theorem 4.2]{Rai11} that the motive $S_{f,a}^{\infty}$ does not depend on the chosen compactification of $f$.
While the proof in \emph{loc. cit.} relies on a computation on a suitable resolution of the singularities of the pair $\big(X,\hat{f}^{-1}(a) \cup F\big)$, this fact will also naturally follow from the results discussed in Section~\ref{section_nonarchimedean}.

\begin{rem}
The motive $S_{\hat{f}-a,U} - i_!S_{f-a}\in \mathcal M_{\hat{f}^{-1}(a)}^{\hat{\mu}}$ can be obtained using motivic integration in the following way. 
Given integers $\delta>0$ and $n\geq 1$, consider the set of arcs
\[
X_{n}^{\delta,\infty}\big(\hat{f}-a\big) = 
		\left\{ \begin{array}{c|c} 
			        \varphi \in \mathcal L(X) \; & 
				\begin{array}{l} \varphi(0) \in F, \: 
				  \ord_F (\varphi) \leq n \delta \\ 
				  \ord \big(\hat{f}-a\big)(\varphi) = n,\: \ac \big(\hat{f}-a\big)(\varphi) = 1
			        \end{array} 
		         \end{array} 
		\right\}, 
\] 
and the zeta function 
\[
Z_{\hat{f}-a,U}^{\delta,\infty}(T) = 
\sum_{n\geq 1} \mes\big(X_{n}^{\delta,\infty}(\hat{f}-a)\big) T^n \in 
\mathcal M_{\hat{f}^{-1}(a)}^{\hat{\mu}}[[T]]\;.
\]
Then, for any $\delta$ large enough, we have
\[
S_{\hat{f}-a,U} - i_!S_{f-a} = 
-\lim_{T\ra \infty} Z_{\hat{f}-a,U}^{\delta,\infty}(T)  \in \mathcal M_{\hat{f}^{-1}(a)}^{\hat{\mu}}\;.
\]
\end{rem}

The motive $S_{f,a}^{\infty}$ can be expected to be a motivic analogue of the invariants $\lambda_a(f)$ discussed in the introduction.
However, no link between the two was established in \cite{Rai11}.
That such a connection exists is a consequence of the next Theorem.

\begin{thm} \phantomsection\label{thm:lambda}
	Let $f$ be a polynomial in $\C[x_1,\dots,x_d]$ and let $a$ be an element of $\A^1_\C(\C)$.
	Then:
	\begin{enumerate} 
		\item \label{thm:lambda_1}
		$
		\chi_{c}\big(S_{f,a}^{\infty}\big) = \chi_c\big(f^{-1}(a_\mathrm{gen})\big) - \chi_c\left(f_!S_{f-a}\right);
		$
		\item \label{thm:lambda_2}
		if $a$ is not a critical value of $f$, then 
			\(
				\chi_{c}\big(S_{f,a}^{\infty}\big) = \chi_c\big(f^{-1}(a_\mathrm{gen})\big) - \chi_c\big(f^{-1}(a)\big)\,;
			\)
			
	        \item \label{thm:lambda_3}
	        if $f$ has isolated singularities in $\A^1_\C$, then
			\(
			\chi_{c}\big(S_{f,a}^{\infty}\big) = (-1)^{d-1}\lambda_a(f)\,.
			\)
\end{enumerate}
\end{thm}

\begin{proof} 
	By applying the Euler characteristic to the definition of $S_{f,a}^{\infty}$ we obtain 
	\[
	\chi_c\big(\hat{f}_{!}S_{\hat{f}-a,U}\big) = \chi_c\left(f_{!}S_{f-a}\right) + 
	\chi_c\left(S_{f,a}^{\infty}\right),
	\]
	therefore we need to establish the following equality:
	\[
	\chi_c\big(\hat{f}_{!}S_{\hat{f}-a,U}\big) = 
	\chi_c\big(f^{-1}(a_\mathrm{gen})\big)\;.
	\]
	By \cite[Proposition 3.17]{GuiLoeMer06a} (or \cite[\S 8]{Bit05a}), as noted in Remark~\ref{remarks_SfU}, for any value $a$ the motive $\hat{f}_{!} S_{\hat{f}-a,U}$ realizes on the class $\big[R_{\hat{f!}}\psi_{\hat{f}-a}(Ri_! \underline{\mathbb Q}_{U})\big]$.
	Using the proper base change for the nearby cycles sheaves via the application $t-a$ in $\mathbb A_{\k}^1$ (see \cite[Proposition 4.2.11]{Dim04a}), we deduce that
	\[
	\big[R_{\hat{f!}}\psi_{\hat{f}-a}(Ri_! \underline{\mathbb Q}_{U})\big] = 
	\big[\psi_{t-a}\big((R\hat{f} \circ Ri)_! \underline{\mathbb Q}_{U}\big)\big]  
		=  \big[\psi_{t-a}(Rf_! \underline{\mathbb Q}_{U})\big]\;.
	\]
	Taking the Euler characteristics $\chi_c$ for the motives and $\chi$ for the sheaves we obtain
	\begin{equation} \label{formule:caracteristique-d-Euler}
	\chi_c\big(\hat{f}_{!}S_{\hat{f}-a,U}\big) = 
	\chi\big([R_{\hat{f!}}\psi_{\hat{f}-a}(Ri_! \underline{\mathbb Q}_{U})]\big) = 
	\chi\big([\psi_{t-a}(Rf_! \underline{\mathbb Q}_{U})]\big) = 
	\chi_c\big(f^{-1}(a_\mathrm{gen})\big)\;,
	\end{equation}
	where the last equality follows from \cite[(2.5.6)]{Kashiwara-Schapira} and \cite[(2.3.26)]{Dim04a}.
	Hence we obtain
	\[
	\chi_{c}(S_{f,a}^{\infty}) = \chi_c\big(f^{-1}(a_\mathrm{gen})\big) -  \chi_c\left(f_{!}S_{f-a}\right),
	\]
	proving part \ref{thm:lambda_1}.
	
	In particular, if $a\in \A^1_\C(\C)$ is not a critical value of $f$, we have $f_{!}S_{f-a} = [f^{-1}(a)]$, and therefore
		\[
		\chi_c\big(f^{-1}(a_\mathrm{gen})\big) - \chi_c\big(f^{-1}(a)\big) = \chi_c\left(S_{f,a}^{\infty}\right),
		\]
	which proves part \ref{thm:lambda_2}.

	Now assume that $f$ has isolated singularities.
	Then the following equalities of motives over $f^{-1}(a)$ hold:
	\begin{align*}
	S_{f-a} & = \big[f^{-1}(a)\setminus \text{crit}(f) \to f^{-1}(a)\big] + \sum_{x \in f^{-1}(a) \cap  \text{crit}(f)} S_{f,x} \\
	& = \big[f^{-1}(a)\to f^{-1}(a)\big] + \sum_{x \in f^{-1}(a) \cap  \text{crit}(f)} \big(S_{f,x}- \big[{x} \to f^{-1}(a)\big]\big)
	\end{align*}
	where, for any critical point $x$, the motive $S_{f,x}$ is the motivic Milnor fiber of $f$ at the point $x$, which is the motive constructed analogously as $S_f$ but only using arcs with origin $x$.
	This can be shown by subdividing the arcs defining $S_{f-a}$ according to whether their origin falls in the smooth part of the fiber or not, see \cite[\S4.4]{Rai11} for more details.
		
	Applying the direct image $f_!$ and taking the Euler characteristic of both sides of the last equation, we obtain 
		\[
		\chi_c(f_! S_{f-a}) = \chi_c\big(f^{-1}(a)\big) + 
		\sum_{x \in f^{-1}(a) \cap  \text{crit}(f)} (-1)^{d-1}\mu(f,x)\;,
		\]
	where $\mu(f,x)$ is the Milnor number of $f$ at $x$, since it follows from results of Denef--Loeser (for instance from \cite[Theorem 4.2.1]{DenLoe98b}) that $\chi_c(S_{f,x}-1) = (-1)^{d-1}\mu(f,x)$.
 
	Finally, by applying to this the formula \ref{eq:lambda} and the previous part of the theorem we obtain the equality 
		\[
		\chi_{c}\big(S_{f,a}^{\infty}\big) = (-1)^{d-1}\lambda_a(f)\;,
		\]
	showing part \ref{thm:lambda_3} and thus concluding the proof of the theorem.
\end{proof}


\subsection{Motivic bifurcation set}

In the light of the previous result, it is natural to define a motivic version of the topological bifurcation set of $f$ as follows.

\begin{defn}[{\cite[Définition 4.6]{Rai11}}]
	The \emph{motivic bifurcation set} of $f$ is the set defined as
	\[
	B_{f}^\mathrm{mot} = \big\{a\in \mathbb A^1_{\k}(k) \;\big|\; S_{f,a}^{\infty} \neq 0\big\} \cup \text{disc}(f)\;,
	\]
where $\text{disc}(f)$ denotes the discriminant of $f$.
\end{defn}

The set $B_{f}^\mathrm{mot}$ is finite, as is proven in \cite[Théorème 4.13]{Rai11} via a computation on a suitable resolution of a compactification $(X,\hat{f})$ of $f$.

If we assume that $\k$ is the field of complex numbers, the first part of Theorem~\ref{thm:lambda} implies the following result about the motivic bifurcation set of $f$.

\begin{cor}\phantomsection\label{thm:inclusion-Btop-Bmot}
	Let $f$ be a polynomial in $\C[x_1,\dots,x_d]$. 
	Then we have
	\begin{equation*}
	\big\{ a \in \A^1_\C(\C)  \;\big|\; 
		\chi_c\big(f^{-1}(a)\big) \neq \chi_c\big(f^{-1}(a_\mathrm{gen})\big)
	\big\}
	\subset
	B_f^\mathrm{mot}\;.
	\end{equation*}
\end{cor}

\begin{proof} 
If $a$ belongs to the discriminant of $f$ then $a$ belongs to $B_f^\mathrm{mot}$ by definition, therefore we can assume that $a$ is not in $\mathrm{disc}(f)$.
Now, since $\chi_c\big(f^{-1}(a)\big)\neq\chi_c\big(f^{-1}(a_\mathrm{gen})\big)$, it follows from the second part of Theorem~\ref{thm:lambda} that $\chi_{c}(S_{f,a}^{\infty})$ is nonzero, therefore $S_{f,a}^{\infty}$ is nonzero as well, that is $a$ belongs to $B_{f}^\mathrm{mot}$.
\end{proof}

\begin{rem}
	In particular, whenever $B_f^\mathrm{top} = \big\{ a \in \A^1_\C(\C)  \;\big|\; 
	\chi_c\big(f^{-1}(a)\big) \neq \chi_c\big(f^{-1}(a_\mathrm{gen})\big)
	\big\}$, then $B_f^\mathrm{top}$ is included in $B_f^\mathrm{mot}$.
	For example, as observed in the introduction, this holds in the case of plane curves, or more generally whenever $f$ has isolated singularities at infinity, by \cite{Par95a}.
	Observe also that Theorem~\ref{thm:lambda} does not require $f$ to have isolated singularities in $\A^d_\C$, therefore it applies also to situations where the invariants $\lambda_a(f)$ are not defined.
\end{rem}

\begin{rem} 
In this remark we explain how to lift the equality \ref{eq:caracteristiquefibregenerique} to an equality in $\mathcal M^{\hat{\mu}}_k$. 
In order to do this we need to recall the notion of global motivic zeta function introduced in \cite{Rai11}. 
For any $n\geq 1$ and $\delta \geq 1$, we consider
\begin{equation*}
	X_{n}^{\mathrm{global},\delta}\big(\hat{f}\big) = 
	   \left\{ \begin{array}{c|c} 
		       \varphi(t) \in \mathcal L(X) & 
		           \begin{array}{l}
			         \ord_F \big(\varphi(t)\big) \leq n \delta, \, \ord \left(\hat{f}\big(\varphi(t)\big) - \hat{f}\big(\varphi(0)\big)\right) = n\\
			         \ac \left(\hat{f}\big(\varphi(t)\big) - \hat{f}\big(\varphi(0)\big)\right) = 1
		           \end{array} 
	           \end{array}
	   \right\}. 
\end{equation*}
	The \emph{global motivic zeta function}, defined as 
	 \[
	 Z_{\hat{f},U}^{\mathrm{global},\delta}(T) = \sum_{n\geq 1} \mes\big(X_{n,a}^{\mathrm{global},\delta}(\hat{f})\big) T^n\;,
	 \]
	is a rational series for $\delta$ large enough (see \cite{Rai11}[Theorem 4.10]), so that we can set
	\[
	S^{\mathrm{global}}_{\hat{f},U} = -\lim_{T\ra \infty} Z_{\hat{f},U}^{\mathrm{global},\delta}(T) \in \mathcal M_{X}^{\hat{\mu}}\,.
	\]	
	
	Moreover, again as in \cite[Theorem 4.10]{Rai11}, if $f$ has finitely many critical points then by decomposing the sets $X_{n}^{\mathrm{global},\delta}$ according to the origin of the arcs we obtain the following decomposition 
	\[
	S_{\hat{f},U}^{\mathrm{global}} = [U \setminus \mathrm{crit}(f) \ra X] + 
	\sum_{x \in \mathrm{crit}(f)} i_{!}S_{f,x} + \sum_{a \in B_f^\mathrm{mot}} \big(S_{\hat{f}-a,U} - i_!S_{f-a}\big)\;,
	\]
	and applying the direct image $\hat{f}_!$ we obtain the following equality in $\mathcal M^{\hat{\mu}}_{\A^1_k}$\;:
		\[
			\hat{f}_!S_{\hat{f},U}^{\mathrm{global}} = \big[U \setminus \mathrm{crit}(f)\to \mathbb A^1_{k}\big] + 
			\sum_{x \in \mathrm{crit}(f)} f_! S_{f,x} + \sum_{a \in B_f^\mathrm{mot}} \hat{f}_{!}\big(S_{\hat{f}-a,U} - i_!S_{f-a}\big) \;.
		\]
	
	We claim that whenever $k$ is the field of complex numbers this equality generalizes the equation \ref{eq:caracteristiquefibregenerique}.
	Indeed, by pushing forward via the structure morphism $p\colon \A^1_\C \to \Spec \, \C$ of $\A^1_\C$ and taking the Euler characteristic we obtain
	\[
	\chi_c\left(p_! \hat{f}_! S_{\hat{f},U}^{\mathrm{global}}\right) = 
	1 + (-1)^{d-1}\sum_{x \in \mathrm{crit}(f)} \mu(f,x) + 
	\sum_{a\in B_f^\mathrm{mot}} \chi_{c}\left(S_{f,a}^{\infty}\right).
	\]

	Now, observe that $\hat{f}_! S_{\hat{f},U}^{\mathrm{\mathrm{global}}} \in \mathcal M_{\mathbb A^1_{\mathbb C}}^{\hat{\mu}}$ is a motive over the affine line $\mathbb A^1_\C$. 
	Its fiber over a value $a$ is the motive $\hat{f}_{!}S_{\hat{f}-a,U}$, hence all fibers have Euler characteristics equal to the Euler characteristic of $f^{-1}(a_\mathrm{gen})$ by the formula \ref{formule:caracteristique-d-Euler}. 
	Since the Euler characteristic of the base $\mathbb A^1_{\mathbb C}$ is equal to $1$, we obtain the equality
	\[
	\chi_{c}\left(p_!\hat{f}_! S_{\hat{f},U}^{\mathrm{global}}\right) = \chi_{c}\big(f^{-1}(a_\mathrm{gen})\big)\;.
	\] 
	Indeed, as in classical topology, if $g \colon V\ra \mathbb A^1$ is a surjective morphism whose fiber have all the same Euler characteristic $c$ then the Euler characteristic of $V$ is $c$ as well. This can be seen using constructible functions and integration against Euler characteristic (see for instance \cite{Viro} or \cite[Rem 4.1.32]{Dim04a})
	or more generally Cluckers--Loeser motivic integration \cite{CluLoe08a}.
	Thus, we obtain
	\[
	\chi_{c}\big(f^{-1}(a_\mathrm{gen})\big) =
	1 + (-1)^{d-1}\sum_{x \in \mathrm{crit}(f)} {\mu(f,x)} + 
	\sum_{a\in B_f^\mathrm{mot}} \chi_{c}\left(S_{f,a}^{\infty}\right)\,,
	\]
	which, since $\chi_{c}\big(S_{f,a}^{\infty}\big) = (-1)^{d-1}\lambda_a(f)$ by part \ref{thm:lambda_3} of Theorem~\ref{thm:lambda}, yields
	\[
	\chi\big(f^{-1}(a_\mathrm{gen})\big) = 1 + (-1)^{d-1}\big(\mu(f) + \lambda(f)\big)\,,
	\]
	which is precisely the identity \ref{eq:caracteristiquefibregenerique} cited in the introduction.
\end{rem}


\section{Analytic nearby fiber at infinity and the Serre bifurcation set}
\label{section_nonarchimedean}

In this section, after recalling some constructions in non-archimedean geometry and motivic integration, we define a non-archimedean analytic version of the nearby fiber at infinity and study its properties.

\subsection{Non-archimedean analytic spaces}

We will briefly recall some basic notions of non-archimedean analytic geometry.
While we chose to adopt the point of view of Berkovich, for the purpose of this paper one could also work with rigid analytic spaces.
We refer the reader to \cite{Nicaise2008} and to the references therein for a more thorough discussion of these theories.

With any $K$-variety $X$ is associated an analytic space over $K$, its \emph{analytification} $X^\an$.
It is a locally ringed space, whose points, in analogy with the theory of schemes, are the morphisms $\Spec \, K' \to X$, where $K'$ is a valued field extension of $K$, modulo the relation that identifies two morphisms $\Spec \, K' \to X$ and $\Spec \, K'' \to X$ if they both factor through a third morphism $\Spec \, K''' \to X$, where $K'''$ is an intermediate valued field extension of both $K|K'$ and $K|K''$.

On the other hand, if $X$ is a separated scheme of finite type over the valuation ring $R$ of $K$, then we can also attach to it a formal scheme over $R$, its \emph{formal completion} $\widehat X$.
As a locally ringed space, $\widehat X$ is isomorphic to the inverse limit of the schemes $X \otimes_R R/(t^n)$; its underlying topological space is $X_k$ and its sheaf of functions is $\varprojlim \mathcal O_{X \otimes_R R/(t^n)}$.	
The \emph{analytic space associated with} $\widehat X$ (sometimes also called the generic fiber of $\widehat X$), denoted by $\widehat X ^\beth$, is the compact subspace of $(X_K)^{\an}$ consisting of those points $x\colon \Spec \, K' \to X_K$ that extend to an $R$-morphism $\tilde x \colon \Spec \, R' \to X$, where $R'$ is the valuation ring of $K'$.
If such an extension exists then it is unique by the valuative criterion of separatedness, therefore we obtain a morphism $\mathrm{sp}_{\widehat X} \colon \widehat X^\beth \to X_k$, called \emph{specialization}, that is defined by sending a point $x$ as above to the image through $\tilde x$ of the closed point of $\Spec \, R'$.

\begin{rem} \phantomsection\label{remarks_nonarchimedeaninclusions}
	Let $X$ be a separated scheme of finite type over $R$ and let $U$ be an open subscheme of $X$. Then the following results follows directly from the definitions:
	\begin{enumerate} 
		\item $(U_K)^{\an}$ is an open subspace of the $K$-analytic space $(X_K)^{\an}$, and we have $$(X_K)^\an \setminus (U_K)^\an = (X_K \setminus U_K)^\an.$$
		\item If $X$ is proper over $R$, then the inclusion of $\widehat X^\beth$ in $(X_K)^\an$ is an isomorphism by the valuative criterion of properness.
		\item The open immersion $U\to X$ induces an isomorphism $\widehat U^\beth\cong \mathrm{sp}_{\widehat X}^{-1}(U)$.
	\end{enumerate}
\end{rem}

\begin{rem} \phantomsection\label{remarks_pointsandarcs}
	Let $X$ be a $k$-variety, let $f\colon X \to \A^1_k = \Spec \, k[t]$ be a morphism, and set $X_R=X\times_{\A^1_k}\Spec\,R$ and $X_K=X\times_{\A^1_k}\Spec\,K$.
	Then the $k$-arcs on $X$ with origin in $f^{-1}(0)$ correspond to the totally ramified points of $\big(\widehat{X_R}\big)^\beth$.
	Indeed, if $\varphi \colon \Spec\,k[[s]] \to X$ is an arc on $X$ with $\ord_f(\varphi)=n>0$, then $\varphi$ factors through $X_K$ and the composition 
	\[
	\Spec\,k[[s]] \stackrel{\varphi}{\longrightarrow} X_K \longrightarrow \Spec\,K
	\]
	identifies $K'=\mathrm{Frac}(k[[s]])\cong K(n)$ with a totally ramified degree $n$ extension of $K$, inducing a $K'$-point $x$ of $X_K^\an$ whose specialization $\mathrm{sp}_{\widehat{X_R}}(x)$ coincides with $\pi_0(\varphi)$, hence a point of $\big(\widehat{X_R}\big)^\beth$.
	Conversely, with a $K(n)$-point of $\big(\widehat{X_R}\big)^\beth$ is associated a morphism $\Spec \,R(d)\to X_K$, whose composition with the projection $X_K\to X$ is a $k$-arc on $X$ such that $\ord_f(\varphi)=n$, as $R(n)\cong k[[t^{1/n}]]$.
	For a more precise statement we refer the reader to \cite[6.1.2]{Nicaise2008} or \cite[9.1.2]{NicSeb07a}.
\end{rem}

More generally, if $\X$ is a (separated and quasi-compact) formal $R$-scheme of finite type, that is a quasi-compact locally ringed space that is locally of the form $\widehat X$ for some separated $R$-scheme of finite type $X$, then by gluing the associated $K$-analytic spaces we obtain a compact $K$-analytic space $\X^\beth$.
We say that $\X$ is a \emph{formal model} of $\X^\beth$, and we say that $\X$ is generically smooth if $\X^\beth$ is a smooth $K$-analytic space.
It follows from a celebrated theorem of Raynaud (see \cite[\S4.10]{Nicaise2008}) that if $X$ is a $K$-variety then every compact subspace of $X^\an$ admits a formal model; in particular, this is true for a distinguished class of compact subspaces of $X^\an$, its affinoid domains.

For the purpose of motivic integration it is sufficient to have a formal model for the unramified part of a given analytic space;
this is formalized by the theory of weak Néron models.
If $X$ is a (separated) $K$-analytic space, a \emph{weak Néron model} of $X$ is a formal scheme of finite type $\X$ over $R$ together with an open immersion $i \colon \X^\beth \to X$ such that for every unramified extension $R'$ of $R$ the map $i$ induces a bijection $\X(R') \to X(K')$, where $K'$ is the fraction field of $R'$.

\subsection{Motivic integration on formal schemes and analytic spaces}
\label{subsection_motivicformal}

Extending on work of Sebag \cite{Seb04a}, Loeser--Sebag \cite{LoeSeb03b} and Nicaise--Sebag \cite{NicSeb07b, NicSeb07a, NicSeb11}, Hartmann \cite{Hartmann} defined a theory of equivariant motivic integration on formal schemes of finite type.
We will briefly summarize the results we need.

Let $n$ be an integer, let $\X$ be a generically smooth and flat formal $R$-scheme of finite type endowed with a good $\mu_n$-action on $\X$ (meaning that every $\mu_n$-orbit is contained in an affine formal subscheme of $\X$), and let $\omega$ be a $\mu_n$-closed \emph{volume form} on the compact $K$-analytic space $\X^\beth$ (that is a nowhere vanishing differential form of degree $\dim \X^\beth$ on $\X^\beth$ satisfying an additional compatibility condition in relation with the $\mu_n$-action, see \cite[Definition 6.2, p.32]{Hartmann}).
Note that this setting includes the case where there is no action, by taking $n=1$; in which case the constructions below reduce to those of \cite{Seb04a, LoeSeb03b, NicSeb07a}.
Then there exists a $\mu_n$-equivariant Néron smoothening $h\colon \mathcal Y \to \X$, that is an equivariant morphism of formal $R$-schemes such that $\mathcal Y$ is a weak Néron model of $\X^{\beth}$ and $h$ factors via an open immersion through an equivariant morphism $\mathcal Y'\to \X$ that induces an isomorphism $\mathcal Y'^{\beth}\to \X^{\beth}$.
It can be shown that the motive
\[
\int_{\mathcal X} \abs{\omega} := 
\sum_{\substack{C \text{ connected}\\\text{component of }\mathcal Y_k}} [C]\mathbb L^{-\ord_{C}\big((h^\beth)^*(\omega)\big)} \in \mathcal M_{\mathcal X_k}^{\mu_n}
\]
only depends on $\mathcal X$ and $\omega$ and not on the Néron smoothening $h$.

Moreover, if $X$ is a smooth $K$-analytic space that admits a weak Néron model $\mathcal U$ (such as for example the space $\X^\beth$ above, with $\mathcal U=\mathcal Y$) and $X$ is endowed with a ${\mu}_n$-action that extends to a good ${\mu}_n$-action on $\mathcal U$, then the image of $\int_{\mathcal U}\abs{\omega}$ under the forgetful morphism $\mathcal M_{\mathcal U_k}^{\mu_n} \to \mathcal M_{k}^{\mu_n}$ only depends on $X$ and $\omega$ and not on $\mathcal U$, and is denoted by
\[
\int_{X} \abs{\omega} \in \mathcal M_{k}^{\mu_n}.
\]

Now let $\mathcal X$ be a generically smooth and flat formal $R$-scheme of finite type, and assume that $\mathcal X^{\beth}$ admits a gauge form $\omega$.
Since weak Néron models only see the unramified points of $\X^\beth$ (which in the setting of Remark~\ref{remarks_pointsandarcs} correspond only to arcs with contact order $1$ along $\X_k$), to see its ramified points
we consider the \emph{volume Poincaré series} of $(\mathcal X, \omega)$, that is defined as 
\[
S(\mathcal X, \omega, T)= \sum_{n\geq 1} \int_{\mathcal X(n)} \abs{\omega(n)} T^n \in \mathcal M^{\hat{\mu}}_{\mathcal X_k}[[T]]\,,
\]
where for any integer $n$ we denote by $\mathcal X(n)=\X \otimes_R R(n)$ the base change of $\mathcal X$ to $R(n) \cong k[[t^{1/n}]]$, endowed with the natural action of $\text{Gal}(K(n)|K)\simeq \mu_n$. 
This series is rational, therefore it admits a limit when $T$ goes to infinity.
The limit, which does not depend on the choice of $\omega$ by \cite[Proposition 8.1]{NicSeb07a}, is called the \emph{motivic volume} of $\X$ and denoted by
\[
\Vol(\mathcal X)= -\lim_{T\to \infty} S(\mathcal X, \omega, T) \in \mathcal M^{\hat{\mu}}_{\mathcal X_k}\,.
\]

Given a smooth connected $k$-variety $U$ endowed with a fixed volume form $\omega$ and a dominant morphism $f \colon U \to \A^1_k$, by taking a motivic volume of the compact $K$-analytic space $\big(\widehat{U_R}\big)^\beth$ one can retrieve the motivic nearby cycles $\mathcal S_f$, see \cite[Theorem 9.13]{NicSeb07a} and \cite[Proposition 7.7]{Hartmann}.
More precisely, $\big(\widehat{U_R}\big)^\beth$ comes endowed with a canonical gauge form, its \emph{Gelfand--Leray form} $\omega/df$, and we have
\[
\Vol\big(\widehat{U_R}\big) = \mathbb L^{-(\dim U-1)}S_f \in \mathcal M_{U_k}^{\hat{\mu}}.
\]
Its direct image by $f_{!}$ does not depend on $\widehat{U_R}$, it is therefore called the \emph{motivic volume of $\Ubeth$}, denoted by $\Vol\left(\Ubeth\right)$.
We have 
\begin{equation} \label{eq:Sf}
\Vol\left(\Ubeth\right) = \mathbb L^{-(\dim U-1)} f_{!} S_f \in \mathcal M_{k}^{\hat{\mu}}\,.
\end{equation}

\subsection{The motivic volume of the analytification of an algebraic variety}
\label{subsection_manu}
Let $U$ be a smooth connected $k$-variety endowed with a fixed volume form $\omega$, let $f \colon U \to \A^1_k$ be a dominant morphism, let $(X,\hat{f})$ be a compactification of $f$, and let $F$ be the closed subset $X\setminus U$.
After seeing the results of the previous section one would expect to be able to retrieve the motive $S_{\hat{f},U}$ from the (non-compact) $K$-analytic space $(X_K)^\an$, since the latter contains $\big(\widehat{U_R}\big)^\beth$ and all the points corresponding to the arcs on $X$ that have origin in $F$ but are not arcs in $F$.
This is indeed the case, as shown by Bultot \cite[\S2.4]{Bultot-Phd}.

We will briefly recall Bultot's construction. 
Without loss of generality, we can assume that the compactification $(X,\hat{f})$ of $f$ is normal, that $F$ is a Cartier divisor on $X$, and that we have additional data $(W_\alpha, g_\alpha)_\alpha$, where $\{W_\alpha\}_\alpha$ is a finite cover of $(X_K)^{\an}$ by affinoid domains and, for each $\alpha$, $g_\alpha$ is an analytic function on $W_\alpha$ such that $W_\alpha\cap (F_K)^\mathrm{an}$ is defined by $g_\alpha=0$ in $W_\alpha$.
For every nonnegative integer $\gamma$, we set 
\[
U_\gamma := \bigcup_{\alpha \in A} \big\{ x \in W_\alpha \;\big|\; \abs{{g}_{\alpha}(x)} \geq \abs{t}^\gamma  \big\}.
\]
The $U_\gamma$ form an increasing sequence of compact analytic domains of $(U_K)^\an$ such that
\[
(U_K)^\an	 = \bigcup_{\gamma \geq 0} U_\gamma = \bigcup_{\alpha \in A} \big\{  x\in W_\alpha \;\big|\; \abs{g_\alpha(x)}>0 \big\}.
\]

In order to see that the motivic volumes of the $U_\gamma$ stabilize, it is useful to compute them relatively to the formal scheme $\widehat{X_R}$ in order to be able to compare them in the Grothendieck ring $\mathcal M_{\mathcal X_k}^{\hat{\mu}}$.
This is possible since, if $\mathcal U$ and $\mathcal U'$ are two weak Néron models of $U_\gamma$ mapping to $\widehat{X_R}$, then there exists a third weak Néron model of $U_\gamma$ dominating both, hence the images of the integral $\int_{\mathcal U} \omega/df$  in $\mathcal M_{\mathcal X_k}^{\hat{\mu}}$ depends only on $U_\gamma$ and $\widehat{X_R}$, but not on the choice of $\mathcal U$.
In particular, the image of the motivic volume $\Vol(\mathcal U)$ in $\mathcal M_{\mathcal X_k}^{\hat{\mu}}$ depends on $U_\gamma$ and not on $\mathcal U$; we denote this motive by $\Vol_{\widehat{X_R}}(U_\gamma)$.

Bultot then proves that there exists an integer $\gamma_0\geq 0$ such that, for every $\gamma\geq \gamma_0$ we have 
\[
\Vol_{\left(\widehat{X_R}\right)}(U_\gamma) = \Vol_{\left(\widehat{X_R}\right)}(U_{\gamma_0}) \in \mathcal M_{X_k}^{\hat{\mu}}\,.
\]
This motive does not depend on the sequence $(U_\gamma)$, we thus denote it by $\Vol_{\widehat{X_R}}\big((U_K)^{\an}\big)$
and call it \emph{motivic volume of $(U_K)^{\an}$ over $\widehat{X_R}$} (mind that this motive is 
called {motivic volume of $U_K$ over $X_k$} in \cite{Bultot-Phd}).
Furthermore, Bultot gives a formula computing $\Vol_{\widehat{X_R}}\big((U_K)^{\an}\big)$ in terms of the combinatorics of a good compactification of $U$, and by comparing it with the analogous formula of Guibert--Loeser--Merle he shows that, as expected, we have 
\begin{equation}\label{eq:SfU}
\Vol_{\widehat{X_R}}\big((U_K)^\an\big) = \mathbb L^{-(\dim U-1)}S_{\hat{f},U} \in \mathcal M_{X_k}^{\hat{\mu}}.
\end{equation}
Since the direct image $\hat{f}_{!}S_{\hat{f},U}$ in $\mathcal M_{k}^{\hat{\mu}}$ does not depend on the choice of $X$, we can define the motivic volume of $(U_K)^{\an}$ as 
\begin{equation*}
	\Vol\big((U_K)^{\an}\big) = \hat{f}_{!}\big(\Vol_{\widehat{X_R}}\big((U_K)^\an\big)\big) \in \mathcal M_{k}^{\hat{\mu}}.
\end{equation*}

\subsection{Analytic nearby fiber at infinity}
\label{subsection_analyticinfinity}

As before, let $U$ be a smooth $k$-variety and let $f\colon U\to\A^1_k$ be a dominant morphism.
We can now introduce non-archimedean analytic version of the {motivic} nearby fibers at infinity of $f$.

\begin{defn} 
	Using these notations, we define the \emph{analytic nearby fiber at infinity of $f$ for the value $0$} to be the non-archimedean $K$-analytic space 
	\[
	\F_{f,0}^{\infty} = (U_K)^{\an} \setminus \Ubeth.
	\]
\end{defn}

Similarly, for every $a$ in $\A^1_k(k)$, we define the \emph{analytic nearby fiber at infinity of $f$ for the value $a$} to be the analytic nearby fiber at infinity of $f-a$ for the value $0$; we denote it by $\F_{f,a}^{\infty}$.
Observe that these analytic spaces are clearly independent of the choice of a compactification of $f$.

\begin{rem}
	The definition of $\F_{f,0}^{\infty}$ is analogous to the one of $S_{f,0}^\infty$, for which one chooses a compactification $(X,\hat f)$ of $(U,f)$ and considers only arcs with origin in $F={X\setminus U}$.
	Indeed, as we observed in Remark~\ref{remarks_pointsandarcs} the correct analogue for the origin of an arc on $X$ is the specialization $\mathrm{sp}_{\widehat{X_R}}(x)$ of a point $x$ of $\big(\widehat{X_R}\big)^\beth$, and the inclusion of $\mathcal F_{f,0}^\infty$ in $\big(\widehat{X_R}\big)^\beth=(X_K)^\an$ induces an isomorphism
		\[
		\mathcal F_{f,0}^\infty \cong \mathrm{sp}_{\widehat{X_R}}^{-1}(F) \setminus (F_K)^\an \;,
		\]
	since we have $(U_K)^\an\cong(X_K)^\an \setminus (F_K)^\an$ by the first part of Remark~\ref{remarks_nonarchimedeaninclusions}, and $\Ubeth \cong \mathrm{sp}_{\widehat{X_R}}^{-1}(U)$ by the third part of the same remark.
	Observe that removing $(F_K)^\an$ is necessary to obtain a space that does not depend on the choice of the compactification $(X,\hat f)$; this independence can also be seen as a consequence of the valuative criterion of properness, as two compactifications can always be dominated by a common third one.
\end{rem}

 We declare the volume of $\F_{f,0}^{\infty}$ to be
\begin{equation*}
	\Vol\left(\F_{f,0}^{\infty}\right) = \Vol\big( (U_K)^{an} \big) - \Vol\left(\Ubeth, \right) \in \mathcal M_{k}^{\hat{\mu}}.
\end{equation*}

\begin{thm} 
Let $U$ be a smooth connected $k$-variety and let $f\colon U\to \A^1_k$ be a dominant morphism.
Then we have an equality
\[
	\Vol(\mathcal F_{f,0}^{\infty}) = \mathbb L^{-(\dim U-1)} S_{f,0}^{\infty} \in \mathcal M_{k}^{\hat{\mu}}
\]
\end{thm}
\begin{proof}
The result follows immediately from the definition of the volume and the equalities \ref{eq:Sf} and \ref{eq:SfU}.
\end{proof}

\begin{rem} \label{rem:HK}
	Another approach to motivic integration was developed by Hrushovski--Kazhdan \cite{HruKaz06} with model theoretic methods and gives a group morphism $\Vol^{HK}$ from the Grothendieck ring of semi-algebraic sets over the valued field $K$ to $\mathcal M_k^{\hat{\mu}}$ (see also \cite[Theorem 2.5.1]{NP}).
	In particular, the volume of any locally closed subset of $(U_K)^\an$ is defined, and the equality $\Vol^{HK}(\mathcal F_{f,0}^\infty) = \Vol^{HK}\big((U_K)^\mathrm{an}\big) - \Vol^{HK}\big(\Ubeth\big)$ holds naturally in this context.
	Moreover, the motivic volumes in \cite{HruKaz06} can be computed in an analogous way as the volumes in \cite{GuiLoeMer06a} and \cite{Bultot} in terms of suitable resolutions of singularities (see \cite[Theorem 2.6.1]{NP}), and this can be used to show that $\Vol^{HK}\big(\Ubeth\big)=f_!S_f$, see \cite[Corollary 2.6.2]{NP} or \cite{Forey}. One should similarly be able to deduce that $\Vol^{HK}\big((U_K)^\mathrm{an}\big)=\hat{f_!}S_{\hat{f},U}$, so that $\Vol^{HK}(\mathcal F_{f,0}^\infty)$ is the volume $\Vol(\mathcal F_{f,0}^{\infty})$ that we defined above.
	Observe that, since our definition of the analytic nearby fiber at infinity is independent of the choice of a compactification of $f$, as is the morphism $\Vol^{HK}$, this approach would also yield a compactification-independent definition of the motivic nearby cyles at infinity.
\end{rem}

\subsection{Serre bifurcation set}
\label{subsection_serrebifurcation}
We will now recall the notion of motivic Serre invariant and use it to define another bifurcation set $B_f^{\mathrm{ser}}$ which we call the Serre bifurcation set of $f$.

Since the Euler characteristic is additive on {the category of} $k$-varieties, it gives rise to a morphism $\chi_c \colon \mathcal M_k^{\hat{\mu}} \to \Z$.
Moreover, since $\chi_c(\mathbb L-1)=\chi_c(\mathbb G_{m,k})=0$, this morphism factors through a morphism
\[
\chi_c \colon \mathcal M_k^{\hat{\mu}}/(\mathbb L-1) \longrightarrow \Z.
\]
Therefore, when interested in working with the Euler characteristic it is often useful to study the class of a variety modulo $\mathbb L-1$.
Now, the constructions made in this section can all be done modulo $\mathbb L-1$, which leads to some simplifications.
Indeed, if $\X$ is a generically smooth and flat formal $R$-scheme of finite type endowed with a $\mu_n$-action, $\omega$ is a $\mu_n$-closed gauge form on $\X^\beth$, and $h\colon \mathcal Y \to \X$ is a $\mu_n$-equivariant Néron smoothening, then in $\mathcal M_k^{\hat{\mu}}$ we have
\[
\int_{\mathcal X} \abs{\omega} \equiv 
\sum_{\substack{C \text{ connected}\\\text{component of }\mathcal Y_k}} [C]\mathbb L^{-\ord_{C}\big((h^\beth)^*(\omega)\big)} \equiv
\sum_{C} [C] =
[\mathcal Y_k]
\,\,\,\,\,\,\,\,\,
\mathrm{ mod }\,\,(\mathbb L-1).
\]
This element of $\mathcal M_k^{\hat{\mu}}/(\mathbb L-1)$, that only depends on $\mathcal X$ and not on $\omega$, is called the \emph{motivic Serre invariant} of $\X$, 
and denoted by $S(\X)$.

\begin{rem} 
	The motivic Serre invariant of a generically smooth formal $R$-scheme of finite type was introduced by Loeser--Sebag \cite{LoeSeb03b}, and developed by Nicaise--Sebag \cite{NicSeb07a, NicSeb07b, NicSeb11}, Bultot \cite{Bultot-Phd} and Hartmann \cite{Hartmann}.
	It generalizes an invariant introduced by Serre in \cite{Serre-65} in order to classify the compact analytic manifolds over a local field $L$ and defined using classical $p$-adic integration with value in the ring $\mathbb Z/(q-1)\mathbb Z$, where $q$ is the cardinality of the residue field $l$ of $L$. 
	Counting  $l$-rational points yields a canonical morphism $\mathcal M_l /(\mathbb L - 1) \to \mathbb Z/(q-1)\mathbb Z$, and Loeser--Sebag showed that the image by this morphism of the motivic Serre invariant $S(X)$ of a smooth and compact $L$-analytic space $X$ is equal to the classical Serre invariant of the underlying compact manifold.
\end{rem}

We then consider the motive
\begin{equation} \label{Serre:formel}
\overline{S}(\mathcal X) := \Vol(\mathcal X) \mod (\mathbb L-1) \in  \mathcal M_{k}^{\hat{\mu}}/(\mathbb L-1)\;,
\end{equation}
which can also be obtained as the limit of the generating series $-\sum_{n\geq 1} S(\mathcal X(n))T^n$.
Similarly, if $X$ is a smooth $K$-analytic space admitting a weak Néron model $\mathcal U$ over $R$, one also sets $S(X)=S(\mathcal U)$, and we obtain a motive
	\begin{equation} \label{Serre:analytic}
\overline{S}(X) = \Vol(X) \mod (\mathbb L-1) \in  \mathcal M_{k}^{\hat{\mu}}/(\mathbb L-1)
\end{equation}
that is also the limit of an analogous generating series.

It follows from the results of Bultot that, if $U, X, \hat{f}, \{U_\gamma\}_\gamma$ are as in subsection~\ref{subsection_manu}, then
\[
	\overline{S}\big({(U_K)^{\an}}\big):=\overline{S}(U_\gamma) \in \mathcal M_{k}^{\hat{\mu}}/(\mathbb L-1)
\]
only depends on $U$ and $f$ and not on $\gamma$, if $\gamma$ is large enough, nor on $X$, $\{U_\gamma\}_\gamma$,  and $\hat{f}$.
We can then define the \emph{Serre invariant} of $\mathcal F_{f,0}^{\infty}$ as
\[
\overline{S}(\mathcal F_{f,0}^{\infty})
 := 
 \overline{S}\big({(U_K)^\an}\big) - \overline{S}\big(\widehat{U_R}\big) 
\in \mathcal M_{k}^{\hat{\mu}}/(\mathbb L-1) \;.
\]

\begin{defn}
Let $U$ be a smooth $k$-variety and let $f\colon U\to \A^1_k$ be a dominant morphism.
The \emph{Serre bifurcation set} of $f$ is  
\[
B_f^{\mathrm{ser}} = \big\{ a \in  \mathbb A^{1}_{\mathbb C} \;\big|\; \overline{S}\left(\mathcal F_{f,a}^{\infty}\right) \neq 0\big\} \cup \text{disc}(f) \,.
\]
\end{defn}

\begin{thm} 
Let $U$ be a smooth $k$-variety and let $f\colon U\to \A^1_k$ be a dominant morphism.
Then we have an equality
\begin{equation*} \label{formule-Serre} 
	\overline{S}(\mathcal F_{f,0}^{\infty}) = S_{f,0}^{\infty} \mod (\mathbb L-1)
\end{equation*}
in $\mathcal M_{k}^{\hat{\mu}}/(\mathbb L-1)$, 
and $B_f^{\mathrm{ser}}$ is contained in $B_f^{\mot}$.
\end{thm}

\begin{proof}
By combining the definition of $\overline{S}(\mathcal F_{f,0}^{\infty})$ and \ref{Serre:formel} and \ref{Serre:analytic} with the formulas \ref{eq:Sf} and \ref{eq:SfU}, we deduce that the two equalities
\[
	\overline{S}\big({(U_K)^\an}\big) = 
\mathbb L^{-(\dim U-1)}\hat{f}_{!}\left(S_{\hat{f},U}\right) \mod (\mathbb L-1)
\]
and 
\[
\overline{S}\left(\widehat{U_R}\right) = 
\mathbb L^{-(\dim U-1)}f_{!}\left(S_f\right) \mod (\mathbb L-1)
\]
hold in $\mathcal M_k^{\hat{\mu}} / (\mathbb L-1)$, from which the theorem follows.
\end{proof}

\begin{cor}  Let $f$ be a polynomial in $\C[x_1,\dots,x_d]$. 
	Then we have the inclusions
	\begin{equation*}
	\big\{ a \in \A^1_\C(\C)  \;\big|\; 
		\chi_c\big(f^{-1}(a)\big) \neq \chi_c\big(f^{-1}(a_\mathrm{gen})\big)
	\big\}
	\subset B_f^{\mathrm{ser}} \subset B_f^\mathrm{mot}\;.
	\end{equation*}

\end{cor}
\begin{proof}
The result follows from from Corollary~\ref{thm:inclusion-Btop-Bmot}, together with the fact that for any value $a$ the Euler characteristic $\chi_c \big(\overline{S}(\mathcal F_{f,a}^{\infty})\big)$ and $\chi_{c}\big(S_{f,a}^{\infty}\big)$ are equal.
\end{proof}

\begin{rem}
	In particular, whenever 
	$B_f^\mathrm{top} = \big\{ a \in \A^1_\C(\C)  \;\big|\; 
	\chi_c\big(f^{-1}(a)\big) \neq \chi_c\big(f^{-1}(a_\mathrm{gen})\big)\big\}$,
	as in the case of plane curves, or more generally whenever $f$ has isolated singularities at infinity, 
	we have a chain of inclusions 
	$B_f^\mathrm{top} \subset B_f^{\mathrm{ser}} \subset B_f^{\mot}$.
\end{rem}

\bibliographystyle{plain}
\bibliography{biblio_complete.bib}

\begin{thebibliography}{10}

\bibitem{SGA7}
{\em Groupes de monodromie en g\'eom\'etrie alg\'ebrique. {II}}.
\newblock Lecture Notes in Mathematics, Vol. 340. Springer-Verlag, Berlin-New
  York, 1973.
\newblock S\'eminaire de G\'eom\'etrie Alg\'ebrique du Bois-Marie 1967--1969
  (SGA 7 II), Dirig\'e par P. Deligne et N. Katz.

\bibitem{ArtLueMel00b}
Enrique Artal~Bartolo, Ignacio Luengo~Velasco, and Alejandro
  Melle-Hern{{\'a}}ndez.
\newblock {Milnor number at infinity, topology and {N}ewton boundary of a
  polynomial function}.
\newblock {\em Math. Z.}, 233(4):679--696, 2000.

\bibitem{Bit05a}
Franziska Bittner.
\newblock {On motivic zeta functions and the motivic nearby fiber}.
\newblock {\em Math. Z.}, 249(1):63--83, 2005.

\bibitem{Bultot}
Emmanuel Bultot.
\newblock Computing zeta functions on log smooth models.
\newblock {\em C. R. Math. Acad. Sci. Paris}, 353(3):261--264, 2015.

\bibitem{Bultot-Phd}
Emmanuel Bultot.
\newblock {\em Motivic integration and logarithmic geometry}.
\newblock PhD thesis, KU Leuven, 2015.
\newblock Available on arXiv (arXiv:1505.05688).

\bibitem{Cas96}
Pierrette Cassou-Nogu\`es.
\newblock Sur la g\'en\'eralisation d'un th\'eor\`eme de {K}ouchnirenko.
\newblock {\em Compositio Math.}, 103(1):95--121, 1996.

\bibitem{CluLoe08a}
Raf Cluckers and Fran\c{c}ois Loeser.
\newblock {Constructible motivic functions and motivic integration}.
\newblock {\em Invent. Math.}, 173(1):23--121, 2008.

\bibitem{DenLoe98b}
Jan Denef and Fran\c{c}ois Loeser.
\newblock {Motivic {I}gusa zeta functions}.
\newblock {\em J. Algebraic Geom.}, 7(3):505--537, 1998.

\bibitem{DenLoe99a}
Jan Denef and Fran\c{c}ois Loeser.
\newblock {Germs of arcs on singular algebraic varieties and motivic
  integration}.
\newblock {\em Invent. Math.}, 135(1):201--232, 1999.

\bibitem{DenLoe01b}
Jan Denef and Fran\c{c}ois Loeser.
\newblock {Geometry on arc spaces of algebraic varieties}.
\newblock In {\em {European {C}ongress of {M}athematics, {V}ol. {I}
  ({B}arcelona, 2000)}}, volume 201 of {\em {Progr. Math.}}, pages 327--348.
  Birkh{\"a}user, Basel, 2001.

\bibitem{DenLoe02a}
Jan Denef and Fran\c{c}ois Loeser.
\newblock {Lefschetz numbers of iterates of the monodromy and truncated arcs}.
\newblock {\em Topology}, 41(5):1031--1040, 2002.

\bibitem{Dim04a}
Alexandru Dimca.
\newblock {\em {Sheaves in topology}}.
\newblock {Universitext}. Springer-Verlag, Berlin, 2004.

\bibitem{Forey}
Arthur Forey.
\newblock Virtual rigid motives of semi-algebraic sets.
\newblock arXiv:1706.07233.

\bibitem{GuiLoeMer05a}
Gil Guibert, Fran\c{c}ois Loeser, and Michel Merle.
\newblock {Nearby cycles and composition with a nondegenerate polynomial}.
\newblock {\em Int. Math. Res. Not.}, (31):1873--1888, 2005.

\bibitem{GuiLoeMer06a}
Gil Guibert, Fran\c{c}ois Loeser, and Michel Merle.
\newblock {Iterated vanishing cycles, convolution, and a motivic analogue of a
  conjecture of {S}teenbrink}.
\newblock {\em Duke Math. J.}, 132(3):409--457, 2006.

\bibitem{Hartmann}
Annabelle Hartmann.
\newblock Equivariant motivic integration on formal schemes and the motivic
  zeta function.
\newblock {\em arXiv:1511.08656}.

\bibitem{HruKaz06}
Ehud Hrushovski and David Kazhdan.
\newblock Integration in valued fields.
\newblock In {\em Algebraic geometry and number theory}, volume 253 of {\em
  Progr. Math.}, pages 261--405. Birkh\"auser Boston, Boston, MA, 2006.

\bibitem{Kashiwara-Schapira}
Masaki {Kashiwara} and Pierre {Schapira}.
\newblock {\em {Sheaves on manifolds. With a short history ``Les d\'ebuts de la
  th\'eorie des faisceaux'' by Christian Houzel.}}
\newblock Berlin etc.: Springer-Verlag, 1990.

\bibitem{Kou76a}
A.~G. Kouchnirenko.
\newblock {Poly{\`e}dres de {N}ewton et nombres de {M}ilnor}.
\newblock {\em Invent. Math.}, 32(1):1--31, 1976.

\bibitem{Loe09a}
Fran\c{c}ois Loeser.
\newblock {Seattle lectures on motivic integration}.
\newblock In {\em {Algebraic geometry---{S}eattle 2005. {P}art 2}}, volume~80
  of {\em {Proc. Sympos. Pure Math.}}, pages 745--784. Amer. Math. Soc.,
  Providence, RI, 2009.

\bibitem{LoeSeb03b}
Fran\c{c}ois Loeser and Julien Sebag.
\newblock {Motivic integration on smooth rigid varieties and invariants of
  degenerations}.
\newblock {\em Duke Math. J.}, 119(2):315--344, 2003.

\bibitem{Loo02a}
Eduard Looijenga.
\newblock {Motivic measures}.
\newblock {\em Ast{\'e}risque}, (276):267--297, 2002.
\newblock S{\'e}minaire Bourbaki, Vol.\ 1999/2000.

\bibitem{Nicaise2008}
Johannes Nicaise.
\newblock Formal and rigid geometry: an intuitive introduction and some
  applications.
\newblock {\em Enseign. Math. (2)}, 54(3-4):213--249, 2008.

\bibitem{NP}
Johannes Nicaise and Sam Payne.
\newblock A tropical motivic {F}ubini theorem with applications to
  {D}onaldson--{T}homas theory, arxiv:1703.10228.

\bibitem{NicSeb07b}
Johannes Nicaise and Julien Sebag.
\newblock Motivic {S}erre invariants of curves.
\newblock {\em Manuscripta Math.}, 123(2):105--132, 2007.

\bibitem{NicSeb07a}
Johannes. Nicaise and Julien Sebag.
\newblock {Motivic {S}erre invariants, ramification, and the analytic {M}ilnor
  fiber}.
\newblock {\em Invent. Math.}, 168(1):133--173, 2007.

\bibitem{NicSeb11}
Johannes Nicaise and Julien Sebag.
\newblock Motivic invariants of rigid varieties, and applications to complex
  singularities.
\newblock In {\em Motivic integration and its interactions with model theory
  and non-{A}rchimedean geometry. {V}olume {I}}, volume 383 of {\em London
  Math. Soc. Lecture Note Ser.}, pages 244--304. Cambridge Univ. Press,
  Cambridge, 2011.

\bibitem{Par88}
Adam Parusi\'nski.
\newblock A generalization of the {M}ilnor number.
\newblock {\em Math. Ann.}, 281(2):247--254, 1988.

\bibitem{Par95a}
Adam Parusi{\'n}ski.
\newblock {On the bifurcation set of complex polynomial with isolated
  singularities at infinity}.
\newblock {\em Compositio Math.}, 97(3):369--384, 1995.

\bibitem{Par97a}
Adam Parusi{\'n}ski.
\newblock {A note on singularities at infinity of complex polynomials}.
\newblock In {\em {Symplectic singularities and geometry of gauge fields
  ({W}arsaw, 1995)}}, volume~39 of {\em {Banach Center Publ.}}, pages 131--141.
  Polish Acad. Sci., Warsaw, 1997.

\bibitem{Pha83a}
Fr{\'e}d{\'e}ric Pham.
\newblock {Vanishing homologies and the {$n$} variable saddlepoint method}.
\newblock In {\em {Singularities, {P}art 2 ({A}rcata, {C}alif., 1981)}},
  volume~40 of {\em {Proc. Sympos. Pure Math.}}, pages 319--333. Amer. Math.
  Soc., Providence, RI, 1983.

\bibitem{Rai11}
Michel Raibaut.
\newblock {Singularit{\'e}s {\`a} l'infini et int{\'e}gration motivique}.
\newblock {\em Bull. Soc. Math. France}, 140(1):51--100, 2012.

\bibitem{Seb04a}
Julien Sebag.
\newblock {Int{\'e}gration motivique sur les sch{\'e}mas formels}.
\newblock {\em Bull. Soc. Math. France}, 132(1):1--54, 2004.

\bibitem{Serre-65}
Jean-Pierre Serre.
\newblock Classification des vari\'et\'es analytiques {$p$}-adiques compactes.
\newblock {\em Topology}, 3:409--412, 1965.

\bibitem{Suz74}
Masakazu Suzuki.
\newblock Propri\'et\'es topologiques des polyn\^omes de deux variables
  complexes, et automorphismes alg\'ebriques de l'espace {${\bf C}^{2}$}.
\newblock {\em J. Math. Soc. Japan}, 26:241--257, 1974.

\bibitem{Tib07a}
Mihai Tib\u{a}r.
\newblock {\em {Polynomials and vanishing cycles}}, volume 170 of {\em
  {Cambridge Tracts in Mathematics}}.
\newblock Cambridge University Press, Cambridge, 2007.

\bibitem{Viro}
Oleg~Ya. Viro.
\newblock Some integral calculus based on {E}uler characteristic.
\newblock In {\em Topology and geometry---{R}ohlin {S}eminar}, volume 1346 of
  {\em Lecture Notes in Math.}, pages 127--138. Springer, Berlin, 1988.

\bibitem{VuiTra84a}
H{\`a}~Huy Vui and L{\^e}~D{\~u}ng Tr{\'a}ng.
\newblock {Sur la topologie des polyn{\^o}mes complexes}.
\newblock {\em Acta Math. Vietnam.}, 9(1):21--32 (1985), 1984.

\end{thebibliography}
\end{document}